\documentclass[11pt]{amsart}
\usepackage{amsmath,amssymb,amsfonts,graphics}

\newtheorem{thm}{Theorem}[section]
\newtheorem{cor}[thm]{Corollary}
\newtheorem{lem}[thm]{Lemma}

\theoremstyle{definition}
\newtheorem{defn}[thm]{Definition}
\theoremstyle{remark}

\numberwithin{equation}{section}

\newcommand{\R}{\mathbb{R}}
\newcommand{\N}{\mathbb{N}}
\newcommand{\Z}{\mathbb{Z}}
\newcommand{\T}{\mathbb{T}}
\newcommand{\C}{\mathbb{C}}
\newcommand{\Cm}{\mathbb{C}^*}
\newcommand{\Rm}{\mathbb{R}_+}

\newcommand{\Zc}{\mathcal{Z}}
\newcommand{\Nc}{\mathcal{N}}
\newcommand{\tw}{\circledast}

\newcommand{\td}{\circledcirc}
\newcommand{\Sm}{\mathcal{S}}
\newcommand{\Zb}{\Zc_{b}^2(G,\Cm)}

\newcommand{\Zbw}{\Zc_{bw}^2(G,\Cm)}
\newcommand{\F}{\mathcal{F}}
\newcommand{\G}{\mathcal{G}}
\newcommand{\Hc}{\mathcal{H}}
\newcommand{\w}{\text{w}^*}

\newcommand{\om}{\omega}
\newcommand{\Om}{\Omega}
\newcommand{\sg} {\sigma}

\newcommand{\supp}{\text{supp}}

\newcommand{\la}{\langle}
\newcommand{\ra}{\rangle}

\begin{document}

\title[Twisted Orlicz algebras, II]
{Twisted Orlicz algebras, II}

\author[Serap \"{O}ztop]{Serap \"{O}ztop}
\address{Department of Mathematics, Faculty of Science, Istanbul University, Istanbul, Turkey}
\email{oztops@istanbul.edu.tr}

\author{Ebrahim Samei}
\address{Department of Mathematics and Statistics, University of Saskatchewan, Saskatoon, Saskatchewan, S7N 5E6, Canada}
\email{samei@math.usask.ca}

\footnote{{\it Date}: \today.



The second named author was supported by NSERC Grant no. 409364-2015 and 2221-Fellowship Program For Visiting Scientists And Scientists On Sabbatical Leave from Tubitak, Turkey.}







\maketitle

\begin{abstract}
Let G be a locally compact group, let $\Om:G\times G\to \Cm$ be a 2-cocycle, and let ($\Phi$,$\Psi$) be a complementary pair of strictly increasing continuous Young functions. We continue our investigation in \cite{OS1} of the algebraic properties of the Orlicz space $L^\Phi(G)$ with respect to the twisted convolution $\tw$ coming from $\Om$. We show that the twisted Orlicz algebra $(L^\Phi(G),\tw)$ posses a bounded approximate identity if and only if it is unital if and only if $G$ is discrete. On the other hand, under suitable condition on $\Om$, $(L^\Phi(G),\tw)$ becomes an Arens regular, dual Banach algebra. We also look into certain cohomological properties of $(L^\Phi(G),\tw)$, namely amenability and Connes-amenability, and show that they rarely happen. We apply our methods to compactly generated group of polynomial growth and demonstrate that our results could be applied to variety of cases.
\end{abstract}

Orlicz spaces are important type of Banach function spaces that are considered in mathematical analysis. Beside the fact that they generalize $L^p$-spaces, they appear
naturally in computation such as the well-known Zygmund space $L \log^+ L$ which is a Banach space related to Hardy-Littlewood maximal functions. They could also contain certain Sobolev spaces as subspaces. Linear properties of Orlicz spaces have been studied thoroughly (see \cite{rao} for example). However their algebraic properties have been left almost untouched possibility due to the fact that they may fail to be an algebra under a natural multiplication. To be more precise, if one consider an Orlicz space $L^\Phi(G)$ associated to the Young function $\Phi$ over
a locally compact group $G$, then one could ask whether the convolution of compactly supported continuous functions on $G$ can be extend to $L^\Phi(G)$.
However, in most cases, this happens if and only if $L^\Phi(G)$ is a subspace of
$L^1(G)$, where the latter condition is rather restrictive; it forces either $G$ to be compact or the Young function $\Phi$ to have a sublinear growth (see \cite{AM}, \cite{HKM}, \cite{S} for details). For example, if $G$ is not compact, then $L^p(G)$ ($1<p<\infty$) can never be an algebra under the convolution.

To compensate for the failure of the having a convolutive structure on $L^p(G)$, we could look at weighted $L^p(G)$ spaces. A {\it weight} $\om$ on $G$ is a locally bounded measurable function from $G$ into the positive reals. For such a weight, one can extend the construction of $L^p(G)$ to the ``weighted" $L_\om^p(G)$, i.e.
$$L_\om^p(G):=\{f: f\om \in L^p(G)\ \text {and}\ \|f\|_\om=\|f\om\|_p \}.$$
These spaces have various properties and numerous applications in harmonic analysis.
For instant, by applying the Fourier transform, we know that Sobolev spaces $W^{k,2}(\T)$ are nothing but certain weighted $l_\om^2(\Z)$ spaces.
A particular aspect of the behavior of weighted $L^p_\om$ spaces over locally compact groups is that they could form an algebra with respect to the convolution! More precisely, when $p=1$ and $\om$ is submultiplicative, it follows routinely that $L_\om^1(G)$ is a Banach algebra.
Even though, this may not hold in general if $p>1$, there are sufficient conditions under which $L_\om^p(G)$ is a Banach algebra with respect to the convolution.
This was first shown by J. Wermer for $G=\R$ in \cite{JW} and
Yu. N. Kuznetsova later extended it to general locally compact groups. She has also studied some important properties of $L_\om^p(G)$ as a Banach algebra such as the existence of an approximate identity and, for an abelian $G$, a description of their the maximal ideal space (see, \cite{K1}, \cite{K2}, and the references therein). Moreover, in \cite{KM} and together with C. Molitor-Braun, they studied other properties such as symmetry, existence of functional calculus and having the Wiener property.

As Orlicz spaces are generalization of $L^p$ spaces, one could also consider weighted Orlicz spaces and study their properties. Very recently, A. Osan\c{c}l{\i}ol  and S. \"{O}ztop have looked at weighted Orlicz spaces over locally compact groups as Banach algebras with respect to convolution (\cite{OO}). They found sufficient conditions for which the corresponding space becomes an algebra and studied their properties. Their work, in part, extend some of the results of Kuznetsova to a wider class of algebras.

In \cite{OS1}, the authors initiated a general approach to study possible algebraic structure on Orlicz spaces related to the convolution multiplication. We considered the twisted convolution $\tw$ coming from a 2-cocycle $\Om$ with values in $\C^*$, the multiplicative group of complex numbers. Sufficient conditions on $\Om$ were found ensuring that the twisted convolution coming from $\Om$ turns the Orlicz space to a
Banach algebra and, for the cases where $|\Om|$ is a 2-coboundary determined by a symmetric submultiplicative weight $\om$, even a Banach $*$-algebra \cite[Theorems 3.3 and 4.5]{OS1}. We called the algebras we obtained the {\it twisted Orlicz algebras}. We applied our method and showed that there are abundant families of symmetric Banach $*$-algebras in the form of twisted Orlicz algebras on compactly generated groups with polynomial growth \cite[Theorems 5.2 and 5.8]{OS1}. Our approach in \cite{OS1} not only embed everything we discuss in the preceding paragraphs but also allows us to systematically and simultaneously study twisted convolution coming from 2-cocycles with values in $\T$ as well as the weighted spaces coming from suitable submultiplicative weights.

In the present manuscript, which is a sequel to \cite{OS1}, we continue our investigation of twisted Orlicz algebras on locally compact groups. We restrict ourselves to Orlicz spaces coming from ($\Phi$,$\Psi$), a complementary pair of strictly increasing continuous Young functions. For a locally compact group $G$ and the 2-cocycle $\Om$, we show that if $(L^\Phi(G),\tw)$ is algebra, then it has a bounded approximate identity if and only if it is unital if and only if $G$ is discrete. On the other hand,
when $G$ is unimodular and $\Om$ satisfies a suitable decomposition criterion, we show that $(L^\Phi(G),\tw)$ is an Arens regular dual Banach algebra. In particular, we show that there exist rich families of Arens regular dual Banach algebra in the form of twisted Orlicz algebras on compactly generated groups with polynomial growth. In these context, we could say that twisted Orlicz algebras behave very similar to weighted group algebras on {\it discrete} groups, an interesting relation that worth further investigations.

We also look at the cohomological properties of twisted Orlicz algebras. When $G$ is non-discrete, we show that $(L^\Phi(G),\tw)$ fails to be amenable or even Connes-amenable in the cases where it is a dual Banach algebra. On infinite discrete groups, we obtain similar results but only when $\Om$ is a 2-coboundary determined by a weight $\om$, i.e. we consider weighted Orlicz algebras $(L^\Phi(G),\tw)\cong (L^\Phi_\om(G),*)$. The general case on discrete groups remains open as our method can not be generalized when $\Om$ has a non-trivial twist.

We finish by pointing out that throughout this paper, we concern ourselves with the
theory ``bounded multiplications" for Banach algebras and Banach modules, as opposed to ``contractive multiplications". Also weights for us are ``weakly submultiplicative" as opposed to ``submultiplicative".

\section{Preliminaries}

In this section, we give some definitions and state some technical results that will be crucial in the rest of this paper. In this paper, $G$ denotes a locally compact group with a fixed left Haar measure $ds$.

\subsection{Orlicz Spaces}

In this section, we recall some facts concerning Young functions and Orlicz spaces. Our main reference is \cite{rao}.

A nonzero function $\Phi:[0,\infty) \to[0,\infty]$ is called a Young
function if $\Phi$ is convex, $\Phi(0)=0$, and $\lim_{x\to \infty} \Phi(x)=\infty$. For a Young function $\Phi$, the complementary function $\Psi$ of
$\Phi$ is given by
\begin{align}\label{Eq:Young function-complementary}
\Psi(y)=\sup\{xy-\Phi(x):x\ge0\}\quad(y\geq 0).
\end{align}
It is easy to check that $\Psi$ is also a Young function. Also, if $\Psi$ is the complementary function of $\Phi$, then $\Phi$ is
the complementary of $\Psi$ and $(\Phi,\Psi)$ is called a
complementary pair. We have the Young inequality
\begin{align}\label{Eq:Young inequality}
xy\le\Phi(x)+\Psi(y)\quad(x,y\ge0)
\end{align}
for complementary functions $\Phi$ and $\Psi$. By our
definition, a Young function can have the value $\infty$ at a point,
and hence be discontinuous at such a point. However, we always consider the pair of complementary Young
functions $(\Phi,\Psi)$ with both $\Phi$ and $\Psi$ being continuous and strictly increasing. In particular, they attain positive values on $(0,\infty)$.

Now suppose that $G$ is a locally compact group with a fixed Haar measure $ds$ and $(\Phi,\Psi)$ is a complementary pair of Young functions. We define
\begin{align}\label{Eq:Orlicz defn-0}
\mathcal{L}^\Phi(G)=\left\{f:G\to\C:f \ \text{is measurable and}\  \int_G\Phi(|f(s)|)\,ds <\infty
\right\}.
\end{align}
Since $\mathcal{L}^\Phi(G)$ is not always a linear space, we define the Orlicz space $L^\Phi(G)$ to be
\begin{align}\label{Eq:Orlicz defn}
L^\Phi(G)=\left\{f:G\to\C:\int_G\Phi(\alpha|f(s)|)\,ds <\infty
\mbox{ for some }\alpha>0\right\},
\end{align}
where $f$ indicates a member in the equivalence classes of measurable functions with respect to the Haar measure $ds$. When $G$ is discrete, we simply use the standard terminology and write $l^\Phi(G)$ instead of $L^\Phi(G)$. Then the Orlicz space is a Banach space under the (Orlicz) norm $\|\cdot\|_\Phi$
defined for $f\in L^\Phi(G)$ by
\begin{align}\label{Eq:Orlicz norm}
\|f\|_\Phi=\sup\left\{\int_G|f(s)v(s)|\,ds: \int_G\Psi(|v(s)|)\,ds \le1\right\},
\end{align}
where $\Psi$ is the complementary function to $\Phi$. One can also
define the (Luxemburg) norm $N_\Phi(\cdot)$ on $L^\Phi(G)$ by
\begin{align}\label{Eq:Orlicz Luxemburg defn}
N_\Phi(f)=\inf\left\{k>0:\int_G\Phi\left(\frac{|f(s)|}{k}\right)
\,ds \le1\right\}.
\end{align}
It is known that these two norms are equivalent; that is,
\begin{align}\label{Eq:Orlicz norm-Luxemburg relation}
N_\Phi(\cdot)\le \|\cdot\|_\Phi\le2 N_\Phi(\cdot)
\end{align}
and
\begin{align}\label{Eq:Orlicz norm-defn relation}
N_\Phi(f)\le1 \ \ \text{if and only if}\ \ \int_G\Phi(|f(s)|)\,ds \le1.
\end{align}
Let $\Sm^\Phi(G)$ be the closure of the linear
space of all step functions in $L^\Phi(G)$. Then $\Sm^\Phi(G)$ is a Banach space
and contains $C_c(G)$, the space of all continuous functions on $G$ with compact support, as a dense subspace \cite[Proposition 3.4.3]{rao}. Moreover, $\Sm^\Phi(G)^*$, the dual of  $\Sm^\Phi(G)$, can be identified with $L^\Psi(G)$ in a natural way \cite[Theorem 4.1.6]{rao}. Another useful characterization of $\Sm^\Phi(G)$ is that
$f\in \Sm^\Phi(G)$ if and only if for every $\alpha>0$, $\alpha f\in \mathcal{L}^\Phi(G)$ \cite[Definition 3.4.2 and Proposition 3.4.3]{rao}.

A Young function $\Phi$ satisfies the $\Delta_2$-condition 
if there
exist a constant $K>0$
such that
\begin{align}\label{Eq:Delta 2 condition}
\Phi(2x)\le K\Phi(x) \ \ \text{for all}\ \ x\ge 0.
\end{align}
In this case we write
$\Phi\in\Delta_2$.
If $\Phi\in\Delta_2$, then it follows that $L^\Phi(G)=\Sm^\Phi(G)$ so that  $L^\Phi(G)^*=L^\Psi(G)$ \cite[Corollary 3.4.5]{rao}. If, in addition, $\Psi\in\Delta_2$, then the Orlicz space $L^\Phi(G)$ is a reflexive Banach space.

As in \cite[Page 20]{rao}, we say that two Young functions $\Phi_1$ and $\Phi_2$ are {\it strongly equivalent} and write $\Phi_1 \approx \Phi_2$ if there exists $0<a\leq b<\infty$ such that $$\Phi_1(ax)\leq \Phi_2(x)\leq \Phi_1(bx) \ \ \ (x\geq 0).$$
It is clear from the definition of the Orlicz space \eqref{Eq:Orlicz defn} that the strongly equivalent Young functions generate the same Orlicz space allowing us to consider different strongly equivalent Young functions to represent the same Orlicz space.

We will frequently use the (generalized) H\"{o}lder's inequality for Orlicz spaces
\cite[Remark 3.3.1]{rao}.
More precisely, for any complementary pair of Young functions $(\Phi,\Psi)$
and any $f\in L^\Phi(G)$ and $g\in L^\Psi(G)$, we have
\begin{align}\label{Eq:Holder inequality}
\|fg\|_1:=\int_G |f(s)g(s)|ds \leq \min\{N_\Phi(f)\|g\|_\Psi , \|f\|_\Phi N_\Psi(g)\}.
\end{align}
This, in particular, implies that $fg\in L^1(G)$.

In general, there is an straightforward method to construct various complementary pairs of strictly increasing continuous Young functions as describes in
\cite[Theorem 1.3.3]{rao}. Suppose that $\varphi: [0,\infty)\to [0,\infty)$ is a continuous strictly increasing function with $\varphi(0)=0$ and
$\lim_{x\to \infty} \varphi(x)=\infty.$
Then $$\Phi(x)=\int_0^x \varphi(y)dy$$ is a continuous strictly increasing Young function and
$$\Psi(y)=\int_0^y \varphi^{-1}(x)dx$$
is the complementary Young function of $\Phi$ which is also continuous and strictly increasing.
Here $\varphi^{-1}(x)$ is the inverse function of $\varphi$. Here are a few families of examples satisfying the above construction (see \cite[Proposition 2.11]{ML} and \cite[Page 15]{rao} for more details):

$(1)$ For $1< p,q<\infty$ with $\frac{1}{p}+\frac{1}{q}=1$, if $\Phi(x)=\frac{x^p}{p}$, then $\Psi(y)=\frac{y^q}{q}$. In this case,
the space $L^{\Phi}(G)$ becomes the Lebesgue space $L^p(G)$ and the norm $\|\cdot\|_{\Phi}$ is equivalent to the
classical norm $\|\cdot\|_{p}$.

(2) If $\Phi(x)=x\ln (1+x)$, then $\Psi(x) \approx \cosh x-1$.

(3) If $\Phi(x)=\cosh x-1$, then $\Psi(x)\approx x\ln (1+x)$.

(4) If $\Phi(x)=e^x-x-1$, then $\Psi(x)=(1+x)\ln(1+x)-x$.


\subsection{2-Cocycles and 2-Cobounaries}

Throughout this article, we use the notation $\Cm$ to denote the multiplicative group of complex numbers, i.e. $\Cm=\C \setminus \{0\}$,
$\Rm$ to be multiplicative group of positive real numbers, and $\T$ to be the unit circle in $\C$.

\begin{defn}\label{D:2-cocycle defn}
Let $G$ and $H$ be locally compact groups such that $H$ is abelian. A {\it (normalized) 2-cocycle on $G$ with values in $H$} is a Borel
measurable map $\Om: G\times G \to H$ such that
\begin{align}\label{Eq:2-cocycle relation}
\Om(r,s)\Om(rs,t)=\Om(s,t)\Om(r,st) \ \ \  (r,s,t \in G)
\end{align}
and
\begin{align}\label{Eq:2-cocycle relation normalization}
\Om(r,e_G)=\Om(e_G,r)=e_H \ \ \ (r\in G).
\end{align}
The set of all normalized 2-cocycles will be denoted by $\Zc^2(G,H)$.
\end{defn}
If $\om: G \to H$ is measurable with $\om(e_G)=e_H$, then it is easy to see that the mapping
$$(s,t)\mapsto \om(st)\om(s)^{-1}\om(t)^{-1}$$
satisfies \eqref{Eq:2-cocycle relation} and \eqref{Eq:2-cocycle relation normalization}. Hence it is a
2-cocycle; such maps are called 2-coboundary. The set of 2-coboundaries will be denoted by $\Nc^2(G,H)$.
It is easy to check that $\Zc^2(G,H)$ is an abelian group under the product
$$\Om_1 \Om_2 (s,t)=\Om_1(s,t)\Om_2(s,t) \ \ (s,t\in G),$$
and $\Nc^2(G,H)$ is a (normal) subgroup of $\Zc^2(G,H)$. This, in particular, implies that
$$\Hc^2(G,H):=\Zc^2(G,H)/\Nc^2(G,H)$$ turns into a group; This is called the 2$^{nd}$ group cohomology of $G$ into $H$
with the trivial actions (i.e. $s\cdot \alpha=\alpha\cdot s=\alpha$ for all $s\in G$ and $\alpha \in H$).

We are mainly interested in the cases when $H$ is $\Cm$, $\Rm$ or $\T$. One essential observation is that
we can view $\Cm=\Rm \T$ as a (pointwise) direct product of groups. Hence, for any 2-cocycle
$\Om$ on $G$ with values in $\Cm$ and $s,t\in G$, we can (uniquely) write
$\Om(s,t)=|\Om(s,t)| e^{i\theta}$ for some $0\leq \theta <2\pi$. Therefore, if we put
\begin{align}
|\Om|(s,t):=|\Om(s,t)| \ \ \text{and}\ \ \Om_\T(s,t):=e^{i\theta},
\end{align}
then $\Om=|\Om|\Om_\T$ (in a unique way) and the mappings $|\Om|$ and $\Om_\T$ are 2-cocycles
on $G$ with values in $\Rm$ and $\T$, respectively.

\subsection{Groups with polynomial growth}\label{S:Groups poly. growth}

Let $G$ be a compactly generated group with a fixed compact symmetric generating neighborhood $U$ of the identity of the group $G$.
$G$ is said to have {\it polynomial growth} if there exist $C>0$ and $d\in \N$ such that for every $n\in \N$
	$$\lambda(U^n)\leq Cn^d \ \ \ (n\in \N).$$
Here $\lambda(S)$ is the Haar measure of any measurable $S\subseteq G$ and
	$$U^n=\{u_1\cdots u_n : u_i\in U, i=1,\ldots, n \}.$$
The smallest such $d$ is called {\bf the order of growth} of $G$ and it is denoted by $d(G)$.
It can be shown that the order of growth of $G$ does not depend on the symmetric generating set $U$, i.e. it is a universal constant for $G$.
Also, by \cite[Lemma 2.3]{FGL}, the compact symmetric neighborhood $U$ can be chosen so that it has a strict polynomial growth, i.e.
there are positive numbers $C_1$ and $C_2$ such that
\begin{align}\label{Eq:stric poly growth}
C_1n^d \leq \lambda(U^n)\leq C_2n^d \ \ \ (n\in \N).
\end{align}

It is immediate that compact groups are of polynomial growth. More generally, every $G$ with the property that the conjugacy class of
every element in $G$ is relatively compact has polynomial growth \cite[Theorem 12.5.17]{Pal}. Also every (compactly generated) nilpotent
group (hence an abelian group) has polynomial growth \cite[Theorem 12.5.17]{Pal}.

Using the generating set $U$ of $G$ we can define a {\it length function} $\tau_U : G \to [0, \infty)$ by
\begin{align}\label{Eq:length function}
\tau_U(x)=\inf \{n\in \N : x\in U^n \} \ \ \text{for} \ \ x \neq e, \ \ \tau_F(e)=0.
\end{align}
When there is no fear of ambiguity, we write $\tau$ instead of $\tau_U$.
It is straightforward to verify that $\tau$ is a symmetric subadditive function on $G$, i.e.
\begin{align}\label{Eq:lenght func-trai equality}
\tau(xy)\leq \tau(x)+\tau(y) \ \ \text{and} \ \ \tau(x)=\tau(x^{-1})\ \ (x,y\in G).
\end{align}
We can use $\tau$ to define various (nontrivial) weights on $G$.
For instance, for every $0< \alpha < 1$, $\beta > 0$, $\gamma >0$, and $C>0$,
we can define the {\it polynomial weight} $\om_\beta$ on $G$ of order $\beta$ by
\begin{align}\label{Eq:poly weight-defn}
\om_\beta(s)=(1+\tau(s))^\beta  \ \ \ \ (s\in G),
\end{align}
and the {\it subexponential weights} $\sg_{\alpha, C}$ and $\rho_{\beta,C}$ on $G$ by
\begin{align}\label{Eq:Expo weight-defn}
\sg_{\alpha,C}(s)=e^{C\tau(x)^\alpha} \ \ \ \ (s\in G)
\end{align}
and
\begin{align}\label{Eq:Expo weight II-defn}
\rho_{\gamma,C}(s)=e^\frac{C\tau(s)}{(\ln (1+\tau(s)))^\gamma} \ \ \ \ (s\in G).
\end{align}


\section{Twisted Orlicz algebras}\label{S:Twisted Orlicz alg}

In this section, we present and summarize what we need from the theory of twisted Orlicz algebras. These are taken from \cite{OS1}.


\begin{defn}\label{D:bound 2 cocycles}
We denote $\Zb$ to be the group of {\bf bounded 2-cocycles on $G$ with values in $\Cm$}
which consists of all element $\Om\in \Zc^2(G,\Cm)$ satisfying the following conditions:\\
$(i)$ $\Om \in L^\infty(G\times G)$;\\
$(ii)$ $\Om_\T$ is continuous.\\
We also define $\Zbw$ to be the subgroup of $\Zb$ consisting of elements $\Om \in \Zb$ for which
$$|\Om|(s,t)=\frac{\om(st)}{\om(s)\om(t)} \ \ \ (s,t\in G),$$ where $\om : G \to \R_+$ is a locally integrable measurable function
with $\om(e)=1$ and $1/\om\in L^\infty(G)$. In this case, we called $\om$ a {\bf weight} on $G$ and say that $|\Om|$ is the {\bf 2-coboundary determined by} $\om$, or alternatively, $\om$ is the {\bf weight associated to} $|\Om|$.
\end{defn}

Now suppose that $\Om\in \Zb$ and $f$ and $g$ are measurable functions on $G$.
We define the {\bf twisted convolution of $f$ and $g$ under $\Om$} to be
\begin{align}\label{Eq:twisted convolution}
f\tw g (t)=\int_G f(s)g(s^{-1}t)\Om(s,s^{-1}t)ds  \ \ \ (t\in G)
\end{align}
It follows routinely that for every $f,g\in L^1(G)$,
$f\tw g\in L^1(G)$ with $\|f\tw g\|_1\leq \|\Om\|_\infty \|f\|_1\|g\|_1$.
We conclude that $(L^1(G),\tw)$ becomes a Banach algebra; it is called
the {\bf twisted group algebra} (see \cite[Section 2]{OS1} for more details).

\begin{defn}
Let $G$ be a locally compact group, let $\Om\in \Zb$, and let $\tw$ be the twisted convolution coming from $\Om$. We say
that $(L^\Phi(G),\tw)$ is a {\bf twisted Orlicz algebra} if $(L^\Phi(G),\tw, \|\cdot\|_\Phi)$ is a Banach algebra, i.e. there is $C>0$ such that
for every $f,g\in L^\Phi(G)$, $f\tw g\in L^\Phi(G)$ with
$$ \|f\tw g\|_\Phi \leq C\|f\|_\Phi \|g\|_\Phi.$$
\end{defn}

In \cite{OS1}, sufficient conditions on $\Om$ were found under which the twisted convolution \eqref{Eq:twisted convolution} turns an Orlicz space to a twisted Orlicz algebra. We formulate it below which is basically \cite[Lemma 3.2 and Theorem 3.3]{OS1}.

\begin{thm}\label{T:twisted Orlicz alg}
Let $G$ be a locally compact group, and let $\Om\in \Zbw$. Then:\\
$(i)$ $L^\Phi(G)$ is a Banach $L^1(G)$-bimodule with respect to the twisted convolution \eqref{Eq:twisted convolution} having $\Sm^\Phi(G)$ as an essential Banach $L^1(G)$-submodule;\\
$(ii)$ Suppose that there exits non-negative measurable functions $u$ and $v$
in $L^\Psi(G)$ such that
\begin{align}\label{Eq:2-cocycle bdd sum}
|\Om(s,t)|\leq u(s)+v(t) \ \ \ (s,t\in G).
\end{align}
Then
for every $f,g\in L^\Phi(G)$, the twisted convolution \eqref{Eq:twisted convolution}
is well-defined on $L^\Phi(G)$
so that $(L^\phi(G),\tw)$ becomes a twisted Orlicz algebra having $\Sm^\Phi(G)$ as a closed subalgebra.
\end{thm}

One instance where our criterion can be successfully applied is when we consider 2-cocycles and Orlicz algebras on compactly generated groups of polynomial growth. In particular, we have the following which is \cite[Corollary 5.3 and Theorem 5.8]{OS1}.

\begin{thm}\label{T:twisted Orlicz alg-poly and exp weight-Poly growth}
Let $G$ be a compactly generated group of polynomial growth, let $\Om\in \Zbw$, and let $\om$ be the weight associated to $\Om$. Then there are $u,v\in \Sm^\Psi(G)$ satisfying \eqref{Eq:2-cocycle bdd sum} so that $(L^\Phi(G),\tw)$ is twisted Orlicz algebras if $\om$ is either one of the following weights:\\
$(i)$ $\om=\om_\beta$, the polynomial weight \eqref{Eq:poly weight-defn} with $1/\om \in \Sm^\Psi(G)$;\\
$(ii)$ $\om=\sg_{\alpha,C}$, the subexponential weight \eqref{Eq:Expo weight-defn};\\
$(iii)$ $\om=\rho_{\gamma,C}$, the subexponential weight \eqref{Eq:Expo weight II-defn}.
\end{thm}

We finish this section by pointing out that in the proof of \cite[Corollary 5.3]{OS1}, it is shown that $1/\om_\beta \in \Sm^\Psi(G)$ if $\beta> \frac{d}{l}$,
where $d:=d(G)$ is the degree of the growth of $G$ and
$l\geq 1$ is such that $\lim_{x\to 0^+}\frac{\Psi(x)}{x^l}$ exists.



\section{Approximate identities}

Let $A$ be a Banach algebra, and let $X$ be a Banach left $A$-module. Recall that a net $(e_\alpha)_{\alpha\in\Lambda}$ in the Banach
algebra $A$ is called a left approximate identity for
$X$ in $A$ if $\lim_{\alpha}\|e_\alpha x-x\|_X=0$ for all $x\in X$. If there exists a
$K>0$ such that $\|e_\alpha\| \le K$ for all $\alpha\in\Lambda$,
then $(e_\alpha)_{\alpha\in \Lambda}$ is called a bounded left
approximate identity for $X$ in $A$. The concepts of (bounded) right approximate identity and (bounded) approximate
identity are defined similarly for Banach right $A$-modules and Banach $A$-bimodules. Since $A$ can be viewed as
Banach $A$-bimodule on itself, these concepts apply to the algebra as well.

In this section, we investigate when twisted Orlicz algebras have (bounded or unbounded) approximate identities and/or when they are unital. To do this and obtain our results, we need to restrict ourselves to the twisted actions coming from 2-cocyles in $\Zbw$. Our results may be compared with the corresponding ones in \cite[Section 4]{OO} where they were obtained for (untwisted) weighted Orlicz algebras with the extra assumption that $L^\Phi(G)\subset L^1(G)$.

The following lemma, which is somewhat straightforward, shows that in many cases, twisted Orlicz algebras can have approximate identities as a module over twisted group algebras.

\begin{lem}\label{L:a.i. Orlicz space}
Let $G$ be a locally compact group, let $\Om\in \Zbw$, and let $\om$ be a weight associated to $|\Om|$. Suppose that
$\{U_i\}_{i\in I}$ be a symmetric neighborhoods of the identity in $G$ with $U_i \searrow \{e\}$ as $i\to \infty$, and put
$e_i=1_{U_i}/\lambda(U_i)$ $(i\in I)$. Then:\\
$(i)$ $\{e_i\om\}_{i\in I}$ is a bounded approximate identity for $(L^1(G),\tw)$;\\
$(ii)$ $\{e_i\om\}_{i\in I}$ is an approximate identity for $(\Sm^\Phi(G),\tw)$.
\end{lem}

\begin{proof}
(i) Let $f\in L^1(G)$ and put $g=f/\om$. For every $t\in G$ and $i\in I$, we have
\begin{eqnarray*}
e_i\om \tw f(t)-f(t) &=& \int_G e_i(s)\om(s)f(s^{-1}t)\Om(s,s^{-1}t)ds-f(t) \\ &=&
\om(t) \int_G e_i(s)(g(s^{-1}t)\Om_\T(s,s^{-1}t)-g(t))ds
\\ &=& \om(t) \int_G e_i(s)(g(s^{-1}t)-g(t))\Om_\T(s,s^{-1}t)ds
\\ &+& \om(t) \int_G e_i(s)g(t)(\Om_\T(s,s^{-1}t)-1) ds.
\end{eqnarray*}
Therefore, by applying Fubini's theorem, we get
\begin{eqnarray*}
\|e_i \om\tw f-f\|_{1} &\leq &
\int_G e_i(s)\int_G \om(t) |g(s^{-1}t)-g(t)| dt ds
\\ &+& \int_G e_i(s) \int_G |f(t)(\Om_\T(s,s^{-1}t)-1)| dt ds \\ &\leq &
\sup_{s\in \supp\,e_i} \|L_sg-g\|_{1,\om} + \|f\|_1\sup_{t\in G, s\in \supp\,e_i} |\Om_\T(s,s^{-1}t)-1|,
\end{eqnarray*}
where the last line approaches 0 as $i\to \infty$. Similarly, we can show that
$\|f\tw e_i \om-f\|_{1}$ as $i\to \infty$.\\
(ii) It follows easily from \cite[Proposition 3.4.3]{rao} that $\{e_i\om\}_{i\in I}\subset \Sm^\Phi(G)$. Let $f\in \Sm^\Phi(G)$. Since $\Sm^\Phi(G)$ is an essential Banach $L^1(G)$-bimodule and $(L^1(G),\tw)$ has a bounded approximate identity, by the Cohen's factorization theorem, there are $g_1,g_2\in L^1(G)$ and $h\in \Sm^\Phi(G)$ such that $f=g_1\tw h\tw g_2$. We have
$$\|e_i\tw f-f\|_{\Phi}\leq \|e_i\tw g_1-g_1\|_{1}\|h\tw g_2\|_{\Phi}\to 0$$
as $i\to \infty$ by (i). Similarly, we can show that $\|f\tw e_i-f\|_{\Phi}\to 0$ as $i\to \infty$.
\end{proof}

The following theorem demonstrate a contrast to the preceding lemma in the sense
that the existence of a bounded approximate identity for the
twisted Orlicz algebras will force the underlying group to be discrete.

\begin{thm}\label{T:orlicz space-bai}
Let $G$ be a locally compact group, let $\Om\in \Zbw$, and let $\om$ be a weight associated to $|\Om|$.
Then the following are equivalent:\\
$(i)$ There is a bounded net $\{f_j\}\subset L^\Phi(G)$ so that for all $g\in C_c(G)$,
$\lim_{j\to \infty} f_j\tw g=g$ in $L^\Phi(G)$;\\
$(ii)$ There is a bounded net $\{f_j\}\subset L^\Phi(G)$ so that for all $g\in C_c(G)$,
$\lim_{j\to \infty}g\tw f_j=g$ in $L^\Phi(G)$;\\
$(iii)$ $L^1 (G)\subseteq L^\Phi(G)$;\\
$(iv)$ $G$ is discrete.
\end{thm}

\begin{proof}
(i)$\Longrightarrow$(iii) By our assumption and Theorem \ref{T:twisted Orlicz alg}(i),
there is $K>0$ such that for every $g\in C_c(G)$,
$$\|g\|_\Phi=\lim_{j\to \infty} \|f_j\tw g\|_\Phi\leq K\|g\|_1\limsup_{j\to \infty} \|f_j\|_\Phi.$$
The final result follows since $C_c(G)$ is norm dense in $(L^1(G),\|\cdot\|_1)$. The part (ii)$\Longrightarrow$(iii) is proven similarly. It is also clear that
(iv)$\Longrightarrow$(i) and (ii). Hence it remains to show (iii)$\Longrightarrow$(iv). Suppose that
(iii) holds. Let $\{e_i\om\}_{i\in I}$ be the bounded approximate identity in $L^1(G)$ for $\Sm^\Phi(G)$ constructed in Lemma \ref{L:a.i. Orlicz space}.
Then, by our assumption, $\{e_i\om\}_{i\in I}$ is bounded in $L^\Phi(G)$. Furthermore, since $1/\om\in L^\infty(G)$,
we have
$$\sup_{i\in I} \|e_i\|_\Phi \leq \|1/\om\|_\infty\sup_{i\in I} \|e_i\om\|_\Phi,$$
implying that $\{e_i\}_{i\in I}$ is bounded in $L^\Phi(G)$. In particular, by \cite[Cororllary 3.4.7]{rao},
$$\sup_{i\in I} 1/\lambda(U_i)[\Phi^{-1}(1/\lambda(U_i))]^{-1}\cong \sup_{i\in I} \|e_i\|_\Phi<\infty,$$
or equivalently, there is a constant $M>0$ such that
\begin{align}\label{Eq:Young function bdd linearly}
\Phi(1/\lambda(U_i))\leq M/\lambda(U_i) \ \ (i\in I).
\end{align}
If $G$ is not discrete, then we can assume that $\lambda(U_i)\to 0$ as $i\to \infty$. For every $x\in \R_+$, pick
$i\in I$ such that $x\lambda(U_i)\leq 1$. Then, it follows from \eqref{Eq:Young function bdd linearly} and the convexity of $\Phi$ that
$$\Phi(x)=\Phi(x\lambda(U_i)/\lambda(U_i))\leq x\lambda(U_i)\Phi(1/\lambda(U_i))\leq Mx.$$
On the other hand, for every $M>0$, the complementary Young function of $\Phi_0(x)=Mx$ is given by $\Psi_0(y)=0$ if $y\leq 1/M$ and $\Psi_0(y)=\infty$ otherwise. Hence, by \cite[Proposition 2.1.2]{rao}, $\Psi$ obtains infinite values so that it is not continuous which contradicts our assumption. Thus $G$ must be discrete, i.e. (iv) holds.
\end{proof}

We are now ready to present the main result of this section which is mostly summarization of what we obtained above.

\begin{thm}\label{T:Orlicz space-bai-non discrete}
Let $G$ be a locally compact group, let $\Om\in \Zbw$, and let $\om$ be a weight associated to $|\Om|$.
If $(L^\Phi(G),\tw)$ is a twisted Orlicz algebra, then the following hold:\\
$(i)$ $(S^\Phi(G),\tw)$ has an approximate identity consisting of compactly supported functions;\\
$(ii)$ For $G$ non-discrete, an approximate identity of $(S^\Phi(G),\tw)$ is also an approximate identity for $(L^\Phi(G),\tw)$ if and only if $L^\Phi(G)=\Sm^\Phi(G)$ if and only if $\Phi\in \Delta_2$;\\
$(iii)$ For $G$ non-discrete, neither $(L^\Phi(G),\tw)$ nor $(S^\Phi(G),\tw)$ has a bounded approximate identity;\\
$(iv)$ Either $(L^\Phi(G),\tw)$ or $(S^\Phi(G),\tw)$ has a unit if and only if
$G$ is discrete.
\end{thm}

\begin{proof}
Part (i) follows from Lemma \ref{L:a.i. Orlicz space} and parts (iii) and (iv) from Theorem \ref{T:orlicz space-bai}.
For part (ii), we first note that an approximate identity of $(S^\Phi(G),\tw)$ is also an approximate identity for $(L^\Phi(G),\tw)$ if and only if $L^\Phi(G)$ is an essential $L^1(G)$-bimodule. Now the final result follows from \cite[Theorem 3.4.14 and Corollary 5.3.2]{rao} (See also discussion in \cite[P. 166]{rao}).
\end{proof}

\section{Dual Banach algebras}\label{S:Dual Banach algebras}

A Banach algebra $A$ is said to be a {\it dual Banach algebra} if there is a closed submodule $\mathcal{B}$ of $A^*$ such that $A=\mathcal{B}^*$ (\cite[Definition 4.4.1]{Run}). It follows routinely that if a Banach algebra $A$ is a dual space, then it is a dual Banach algebra if and only if multiplication on $A$ is separately $w^*$-continuous. The class of dual Banach algebras includes
von Neumann algebras, measure algebras of locally compact groups, and
$B(E)$, where $E$ is a reflexive Banach space \cite[Examples 4.4.2]{Run}.
In fact, it is shown in \cite[Corollary 3.8]{Daws} that dual Banach algebras are precisely $w^*$-closed subalgebra of $B(E)$, where $E$ is a reflexive Banach space.

In this section, we investigate under what conditions twisted Orlicz algebras
are dual Banach algebras. We will show that this phenomenon occur in majority
of the cases when know we can get an algebra. We start with the following lemma
which shows that whenever $L^\Phi(G)$ is an algebra, there are natural $L^\Phi(G)$-bimodule actions on $L^\Psi(G)$ coming from module actions of
$L^\Phi(G)$ on $L^\Phi(G)^*$.

\begin{lem}\label{L:twisted Orlicz alg-module over dual}
Let $G$ be a locally compact group, and let $\Om\in \Zb$. Suppose that $(L^\Phi(G),\tw)$ is a twisted Orlicz algebra. Then $L^\Psi(G)$ is a Banach $L^\Phi(G)$-bisubmodule of $L^\Phi(G)^*$ so that the twisted actions of $L^\Phi(G)$ on $L^\Phi(G)$ are defined by $(g\in L^\Phi(G)$, $h\in L^\Psi(G)$, $s\in G)$
\begin{align}\label{Eq:twisted mod actions}
g\td h (s)=\int_G g(t)h(st)\Om(s,t)dt  \ \ , \ \ h\td g (s)=\int_G g(t)h(ts)\Om(t,s)dt.
\end{align}
Moreover, for every $f,g\in L^\Phi(G)$ and $h\in L^\Psi(G)$, we have
\begin{align}\label{Eq:twisted convolution vs twisted mod actions}
\la f \tw g , h \ra=\la f , g\td h \ra=\la  g , h\td f \ra.
\end{align}
\end{lem}

\begin{proof}
Let $g\in L^\Phi(G)$ and $h\in L^\Psi(G)$. Since $\Phi$ and $\Psi$ are continuous and positive on $\R^+$, both $g$ and $h$ vanishes outside of a $\sg$-finite measurable subset of $G$ so that the functions $g\td h$ and $h\td g$ defined in
\eqref{Eq:twisted mod actions} are measurable. Moreover, a straightforward computation shows that for every $f\in L^\Phi(G)$, \eqref{Eq:twisted convolution vs twisted mod actions} holds. Hence, in particular,
\begin{eqnarray*}
|\la f , g\td h \ra| &=& |\la f \tw g , h \ra| \\
&\leq& \|f\tw g\|_\Phi N_\Psi(h) \ \ \ (\text{by}\ \eqref{Eq:Holder inequality})\\
&\leq& C\|f\|_\Phi \|g\|_\Phi N_\Psi(h)\\
&\leq& 2CN_\Phi(f)\|g\|_\Phi N_\Psi(h).
\end{eqnarray*}
Thus, by \eqref{Eq:Orlicz norm} and \eqref{Eq:Orlicz norm-defn relation}, $\|g\td h\|_\Psi \leq 2C\|g\|_\Phi N_\Psi(h)$
implying that $L^\Psi(G)$ is a Banach left $L^\Phi(G)$-module with respect to the twisted actions $\td$. Similarly we can show that it is a right $L^\Phi(G)$-module.
\end{proof}

In lieu of the preceding Lemma and the fact that $L^\Phi(G)=\Sm^\Psi(G)^*$,
in order to show that $L^\Phi(G)$ is a dual Banach algebra, it suffices to show that $\Sm^\Psi(G)$ is a closed under the actions \eqref{Eq:twisted mod actions}
of $L^\Phi(G)$. We will show that under some reasonable conditions, this indeed happens.

\begin{thm}\label{T:twisted Orlicz alg-dual Ban alg}
Let $G$ be a locally compact unimodular group, and let $\Om \in \Zb$. Suppose that there exit non-negative measurable functions $u,v \in \Sm^\Psi(G)$ satisfying \eqref{Eq:2-cocycle bdd sum}. Then $\Sm^\Psi(G)$ is an essential Banach $L^\Phi(G)$-bimodule with respect to the twisted actions $\td$ defined in \eqref{Eq:twisted mod actions}. In particular, $(L^\Phi(G),\tw)$ is a dual Banach algebra.
\end{thm}

\begin{proof}
We first note that by Theorem \ref{T:twisted Orlicz alg}, $(L^\Phi(G),\tw)$ is a twisted Orlicz algebra.
Let $g\in L^\Phi(G)$ and $h\in \Sm^\Psi(G)$. By Lemma \ref{L:twisted Orlicz alg-module over dual}, $g\td h\in L^\Psi(G)$.
It follows from \eqref{Eq:2-cocycle bdd sum} that for every $s\in G$,
\begin{eqnarray*}
|g\td h|(s) &\leq& \int_G |g(t)h(st)|u(s)dt + \int_G |g(t)h(st)|v(t)dt
\\ &\leq & u(s) N_\Phi(g)\|\delta_{s^{-1}}*h\| _\Psi+ \int_G |g(t)h(st)|v(t)dt \ \ (\text{by} \ \eqref{Eq:Holder inequality})
 \\ &= & u(s) N_\Phi(g)\|h\|_\Psi+ (h*(\check{g}\check{v})(s),
\end{eqnarray*}
where $\check{k}(s)=k(s^{-1})$. Thus
$$|g\td h|\leq \|g\|_\Phi\|h\|_\Psi u+ h*(\check{g}\check{v}).$$
By our hypothesis, $u\in \Sm^\Psi(G)$. On the other hand, since $gv\in L^1(G)$ and $G$ is unimodular, $\check{g}\check{v}\in L^1(G)$. Hence $h*\check{g}\check{v}\in \Sm^\Psi(G)$ as
$\Sm^\Psi(G)$ is an $L^1(G)$-bimodule under convolution.
The final results follows since $|g\td h|\leq f$ with $f\in \Sm^\Psi(G)$ implies that $g\td h\in \Sm^\Psi(G)$. Similarly we can show that
$h\td g\in \Sm^\Psi(G)$. In particular, $(L^\Phi(G),\tw)$ is a dual Banach algebra. 
\end{proof}

The following corollary which is an immediate consequence of Theorem \ref{T:twisted Orlicz alg-dual Ban alg} and Theorem \ref{T:twisted Orlicz alg-poly and exp weight-Poly growth} demonstrates that the conditions in Theorem \ref{T:twisted Orlicz alg-dual Ban alg} are general enough as they can be applied to large classes of twisted Orlicz algebras on compactly generated group of polynomial growth.

\begin{cor}\label{C:twisted Orlicz alg-poly and exp weight-dual Ban alg}
Let $G$ be a compactly generated group of polynomial growth, let $\Om\in \Zbw$,
and let $\om$ be the weight associated to $\Om$. Then $(L^\Phi(G),\tw)$ is a dual Banach algebra if $\om$ is either of the following weights:\\
$(i)$ $\om=\om_\beta$, the polynomial weight \eqref{Eq:poly weight-defn} with $1/\om \in \Sm^\Psi(G)$;\\
$(ii)$ $\om=\sg_{\alpha,C}$, the subexponential weight \eqref{Eq:Expo weight-defn};\\
$(iii)$ $\om=\rho_{\gamma,C}$, the subexponential weight \eqref{Eq:Expo weight II-defn}.



\end{cor}



\section{Arens regularity}\label{S:Arens regularity}

Let $A$ be a Banach algbera, and let $A^{**}$ be its second dual. There are  standard methods to extend the multiplication on $A$ to $A^{**}$ in two ways.
These multiplications are known as the {\it first} and {\it second Arens products} on $A^{**}$ and are denoted by $\Box$ and $\diamond$, respectively. We say that $A$ is {\it Arens regular} if the first and second Arens products on $A^{**}$ coinsides (See, for example, \cite[Definition 2.6.16 and Eq (2.6.27)]{D}). It is well-known that C$^*$-algebras are Arens regular whereas the group algebra or the measure algebra of an infinite locally compact groups are never Arens regular \cite[Theorem 3.3.28]{D}. It is also shown in \cite{CY} that a weighted group algebra is Arens regular if and only if the underlying group is discrete and countable and the weight satisfies certain decay condition (See also \cite[Theorem 8.11]{DL})

Our goal in this section is to extend the method of \cite[Theorem 8.11]{DL} to
obtain Arens regularity of twisted Orlicz algebras. We show that in contract to the result of \cite{CY}, twisted Orlicz algebras can be Arens regular for non-discrete groups. Since $L^\Phi(G)=\Sm^\Psi(G)^*$, we have $$L^\Phi(G)^{**}=L^\Phi(G)\oplus \Sm^\Psi(G)^\bot,$$ where
$$\Sm^\Psi(G)^\bot=\{\F \in L^\Phi(G)^{**} : F=0 \ \text{on}\ \Sm^\Psi(G) \}.$$
Our method is to use the condition given in Theorem \ref{T:twisted Orlicz alg-dual Ban alg} to break down the twisted convolution into two separate parts
and then show that first and second Arens products vanishes on $\Sm^\Phi(G)^\perp$. This, in particular, implies that $L^\Phi(G)$ is Arens regular.

\begin{lem}\label{L:map Orlicz space to bounded functions}
Let $G$ be a locally compact unimodular group, and $L\in L^\infty(G\times G)$. There are bounded continues operators
\begin{align*}
\xi,\eta:L^\Phi(G)\times L^\Phi(G)^* \to L^\infty(G)
\end{align*}
such that for every $g\in L^\Phi(G)$ and $h\in L^\Psi(G)$, we have
\begin{align}\label{Eq:map Orlicz space to bounded functions-I}
\xi(g,h)(s)=\int_G g(t)h(st)L(s,t) dt \ \ (s\in G),
\end{align}
and
\begin{align}\label{Eq:map Orlicz space to bounded functions-II}
\eta(g,h)(s)=\int_G g(t)h(ts)L(t,s) dt \ \ (s\in G).
\end{align}
Moreover for every $g\in L^\Phi(G)$, both operators $\xi(g,\cdot)$ and $\eta(g,\cdot)$ are $\w-\w$ continuous.
\end{lem}

\begin{proof}
Without loss of generality, we can assume that $\|L\|_\infty=1$. For every $g\in L^\Phi(G)$ and $h\in L^\Psi(G)$, define
$$\xi(g,h)(s)=\int_G g(t)h(st)L(s,t) dt$$
 and $$\eta(g,h)(s)=\int_G g(t)h(ts)L(t,s) dt.$$
Clearly, $\xi(g,h)$ and $\eta(g,h)$ are measurable. Moreover
$$|\xi(g,h)|\leq |h|*|\check{g}| \ \ \text{and} \ \ |\eta(g,h)|\leq |h|*|g|,$$
and so, both $\xi(g,h)$ and $\eta(g,h)$ belong to $L^\infty(G\times G)$ with
$$\max\{\|\xi(h,h)\|_\infty,\|\eta(g,h)\|_\infty \}\leq \|h\|_\Psi N_\Phi(g).$$
For every $f\in L^1(G)$, we have
\begin{align}\label{Eq:1}
\int_G f(s)\xi(g,h)(s)ds=\int_G h(t) \zeta(f,g)(t) dt
\end{align}
where
$$\zeta(f,g)(t)=\int_G f(s)g(s^{-1}t)L(s,s^{-1}t) ds.$$
Since
$|\zeta(f,g)|\leq  |f|*|g|$
and $|f|*|g|\in L^\Phi(G)$ (as $L^\Phi(G)$ is a Banach $L^1(G)$-bimodule under convolution), it follows that $\zeta(f,g)\in L^\Phi(G)$.
Thus if we take $\Hc\in L^\Phi(G)^*=\Sm^\Psi(G)^{**}$ and $\{h_k\}\subset \Sm^\Psi(G)$ with $h_k\to^{w^*} \Hc$, then it follows from \eqref{Eq:1} that
$$\lim_{k\to \infty}  \int_G f(s)\xi(g,h_k)(s)ds=\lim_{k\to \infty} \int_G h_k(t) \zeta(f,g)(t) dt=\la \Hc , \zeta(f,g) \ra.$$
Therefore we can define $\xi(g,\Hc)\in L^\infty(G)$ to be the function satisfying \begin{align}\label{Eq:2}
 \la f , \xi(g,\Hc) \ra=\la \Hc , \zeta(f,g) \ra  \ \ (f\in L^1(G),g\in L^\Phi(G)).
 \end{align}
 It is straightforward to verify that $\xi:L^\Phi(G)\times L^\Phi(G)^* \to L^\infty(G)$ is a bounded bilinear map. Also the relation \eqref{Eq:2} implies that for every $g\in L^\Phi(G)$, $\xi(g,\cdot)$ is $\w-\w$ continuous. Similarly, we can prove the desired statement for $\eta$.
\end{proof}

\begin{thm}
Let $G$ be a locally compact unimodular group, and let $\Om \in \Zb$. Suppose that there exit non-negative measurable functions $u,v \in \Sm^\Psi(G)$ satisfying \eqref{Eq:2-cocycle bdd sum}. Then there are bounded bilinear operators
\begin{align}\label{Eq:second dual action Orlicz space-splitting operator}
\xi,\eta:L^\Phi(G)\times L^\Phi(G)^* \to L^\infty(G)
\end{align}
such that for every $f,g\in L^\Phi(G)$ and $\Hc\in L^\Phi(G)^*$,
\begin{align}\label{Eq:second dual action Orlicz space-splitting relation}
\la f\tw g , \Hc \ra =\la f , u\xi(g,\Hc) \ra+\la g , v\eta(f,\Hc) \ra.
\end{align}
Moreover, both operators $\xi(g,\cdot)$ and $\eta(g,\cdot)$ are $\w-\w$ continuous.
\end{thm}

\begin{proof}
By our hypothesis and \eqref{Eq:2-cocycle bdd sum}, we can take an element $L\in L^\infty(G\times G)$ with $\|L\|_\infty\leq 1$ such that
 $$\Om(s,t)=L(s,t)\Big(u(s)+v(t)\Big) \ \ \ (s,t\in G).$$
 Suppose that $\xi$ and $\eta$ are the mapping defined in Lemma \ref{L:map Orlicz space to bounded functions} associated to $L$. It particular, $\xi$ and $\eta$ satisfy \eqref{Eq:map Orlicz space to bounded functions-I} and
\eqref{Eq:map Orlicz space to bounded functions-II}, respectively, and the operators $\xi(g,\cdot)$ and $\eta(g,\cdot)$ are w$^*$-w$^*$ continuous. Take $\Hc\in L^\Phi(G)^*=\Sm^\Psi(G)^{**}$ and $\{h_k\}\subset \Sm^\Psi(G)$ with $h_k\to^{w^*} \Hc$. For every $f,g\in L^\Phi(G)$, we have
\begin{eqnarray*}
\la f \tw g , \Hc \ra &=& \lim_{k\to \infty}  \la f \tw g , h_k  \ra \\ &=&
 \lim_{k\to \infty}  \int_G\int_G f(s) g(t)h_k(st)\Om(s,t) dsdt \\
 &=& \lim_{k\to \infty} [\int_G f(s)u(s)\xi(g,h_k)(s)ds+\int_G g(t)v(t)\eta(f,h_k)(t) dt]\\
 &=& \int_G f(s)u(s)\xi(g,\Hc)(s)ds+\int_G g(t)v(t)\eta(f,\Hc)(t) dt \\
 &=& \la f , u\xi(g,\Hc) \ra+\la g , v\eta(f,\Hc) \ra.
\end{eqnarray*}
In the above, we have used the facts that $fu,gv\in L^1(G)$, both $\xi(g,\cdot)$ and $\eta(f,\cdot)$ are $\w-\w$ continuous and $\Sm^\Psi(G)$ is a Banach $L^\infty(G)$-module under pointwise product.
\end{proof}

We are now ready to obtain our desired result. We recall the formula \cite[Eq (2.6.28)]{D} which gives a useful formulation of the first and second Arens products. For every $\F,\G\in A^{**}$ and the nets $\{f_i\}$, $\{g_j\}$ in $A$
for which $f_i\to \F$ and $g_j\to \G$ in the $w^*$-topology $\sg(A^{**},A^*)$, we have
\begin{align}
\F \Box \G =w^*-\lim_{i\to \infty} w^*-\lim_{j\to \infty} f_i g_j \ \ \text{,}
\ \ \F \diamond \G =w^*-\lim_{j\to \infty} w^*-\lim_{i\to \infty} f_i g_j.
\end{align}

\begin{thm}\label{T:twisted Orlicz alg-Arens regular}
Let $G$ be a locally compact unimodular group, and let $\Om \in \Zb$. Suppose that there exit non-negative measurable functions $u,v \in \Sm^\Psi(G)$ satisfying \eqref{Eq:2-cocycle bdd sum}. Then
$$ \F \Box \G=\F \diamond \G=0 \ \ (\F,\G \in \Sm^\Psi(G)^\perp).$$
In particular, $(L^\Phi(G),\tw)$ is Arens regular.
\end{thm}

\begin{proof}
Take $\F,\G \in \Sm^\Psi(G)^\perp$. We will show that $\F\Box \G=0$. The proof of $\F\diamond \G=0$ will be similar. Let $\{g_j\}$ be a net in
$L^\Phi(G)$ such that $w^*-\lim_{j\to \infty} g_j=\G$. Suppose that $\xi$ and $\eta$ are the operators \eqref{Eq:second dual action Orlicz space-splitting operator} satisfying \eqref{Eq:second dual action Orlicz space-splitting relation}. For every $f\in L^\Phi(G)$, we have
\begin{eqnarray*}
\la f \Box \G , \Hc \ra &=& \lim_{j\to \infty} \la f \tw g_j , \Hc \ra \\
 &=& \lim_{j\to \infty} [\la fu , \xi(g_j,\Hc) \ra+\la g_j , v\eta(f,\Hc) \ra]  \\
 &=& \lim_{j\to \infty} \la fu , \xi(g_j,\Hc) \ra+\la \G , v\eta(f,\Hc) \ra \\
 &=& \lim_{j\to \infty} \la fu , \xi(g_j,\Hc) \ra,
 \end{eqnarray*}
since $v\eta(f,\Hc)\in \Sm^\Psi(G)$ and $\G\in \Sm^\Psi(G)^\perp$.
Now let $I$ be the $\|\cdot\|_1$-closure of $L^\Phi(G)u$ in $L^1(G)$. The preceding computation, together with the fact that $\{g_j\}$ can be chosen to be uniformly bounded in $L^\Phi(G)$, shows that the mapping
$$ I \ni h\mapsto \lim_{j\to \infty} \la h , \xi(g_j,\Hc) \ra$$
is well-defined bounded linear functional on $I$. Hence, by the Hahn-Banach theorem, it has an extension to $L^1(G)^*=L^\infty(G)$; we denote it by $\xi(\G,\Hc)$. In particular,
$$ \lim_{j\to \infty} \la fu , \xi(g_j,\Hc) \ra=\la fu , \xi(\G,\Hc) \ra \ \ \ \ (f\in L^\Phi(G)).$$
Therefore, if we take a net $\{f_i\}\subset L^\Phi(G)$ such that $w^*-\lim_{i\to \infty} f_i=\F$, then
\begin{eqnarray*}
\la \F \Box \G , \Hc \ra &=& \lim_{i\to \infty} \la f_i \Box \G , \Hc \ra \\
 &=& \lim_{i\to \infty} \lim_{j\to \infty} \la f_iu , \xi(g_j,\Hc) \ra  \\
 &=& \lim_{i\to \infty} \la f_iu , \xi(\G,\Hc) \ra \\
 &=& \lim_{i\to \infty} \la f_i , u\xi(\G),\Hc) \ra\\
 &=& \la \F , u\xi(\G,\Hc) \ra \\
 &=& 0,
\end{eqnarray*}
since $u\xi(\G,\Hc)\in \Sm^\Psi(G)$ and $\F\in \Sm^\Psi(G)^\perp$. This completes the proof.
\end{proof}

The following can be compared with Corollary \ref{C:twisted Orlicz alg-poly and exp weight-dual Ban alg}. Its proof follows from the preceding theorem and Theorem \ref{T:twisted Orlicz alg-poly and exp weight-Poly growth}.

\begin{cor}\label{C:twisted Orlicz alg-poly and exp weight-Arens regular}
Let $G$ be a compactly generated group of polynomial growth, let $\Om\in \Zbw$,
and let $\om$ be the weight associated to $\Om$. Then $(L^\Phi(G),\tw)$ is Arens regular if $\om$ is either of the following weights:\\
$(i)$ $\om=\om_\beta$, the polynomial weight \eqref{Eq:poly weight-defn} with $1/\om \in \Sm^\Psi(G)$;\\
$(ii)$ $\om=\sg_{\alpha,C}$, the subexponential weight \eqref{Eq:Expo weight-defn};\\
$(iii)$ $\om=\rho_{\gamma,C}$, the subexponential weight \eqref{Eq:Expo weight II-defn}.
\end{cor}

We finish this section by pointing out that if $A$ is an Arens regular Banach algebra, then $A^{**}$ is a dual Banach algebra \cite[Example 4.4.2]{Run}. Hence under the assumption of Theorem \ref{T:twisted Orlicz alg-Arens regular}, $L^\Phi(G)^{**}$ is a dual Banach algebra.

\section{Cohomological properties}

In this section, we consider two important cohomological properties, namely amenability and Connes amenability, that are associated to Banach algebras and dual Banach algebras, respectively. We will show that neither of them can hold on most twisted Orlicz algebras.

\subsection{Amenability}

A Banach algebra $A$ is amenable if every bounded derivation from $A$ into the dual of any Banach $A$-bimodule is inner \cite[Definition 2.1.9]{Run}. By the celebrated theorem of Johnson, the group algebra of a locally compact group is amenable if and only if the underlying group is amenable \cite[Theorem 2.1.8]{Run}. In contrast, the weighted group algebra of an amenable group is amenable if and only if the weight is diagonally bounded. For a symmetric weight, the later condition implies the weight must be trivial.

We now look at the amenability property for a twisted Orlicz algebra and show that in most cases, its amenability implies that it must be finite dimensional. We start with the following result which is an immediate consequence of Theorem \ref{T:Orlicz space-bai-non discrete} and the fact that a necessary condition for a Banach algebra to be amenable is to have a bounded approximate identity
\cite[Proposition 2.2.1]{Run}.

\begin{thm}\label{T:Orlicz space-amen-non discrete}
Let $G$ be a non-discrete locally compact group, and let $\Om\in \Zbw$.
If $(L^\Phi(G),\tw)$ is a twisted Orlicz algebra, then neither $(L^\Phi(G),\tw)$ nor $(S^\Phi(G),\tw)$ are amenable.
\end{thm}

When $G$ is discrete, one can not apply the criterion \cite[Proposition 2.2.1]{Run} as the twisted orlicz algebra is unital. To compensate for this, we use the concept of the augmentation ideal. However, this method only works when twisted actions comes from a 2-coboundary determined by a weight.
i.e. when we consider the {\it weighted Orlicz algebras}. Let us first recall some terminology (see also \cite[Lemma 3.3]{OS1}).


Let $G$ be a locally compact group, and let $\om$ be a weight on $G$. We define
the {\it weighted $L^\Phi$-space}
\begin{align}
L^\Phi_\om(G):=\{ f:G \to \C : f\om \in L^\Phi(G)\}.
\end{align}
and the {\it weighted $\Sm^\Phi$-space}
\begin{align}
\Sm^\Phi_\om(G):=\{ f:G \to \C : f\om \in \Sm^\Phi(G)\}.
\end{align}
Then both $L^\Phi_\om(G)$ and $\Sm^\Phi_\om(G)$ with the norm $\|f\|_{\Phi,\om}=\|f\om\|_\Phi$ are Banach spaces.
Moreover, $(L^\Phi_\om(G),\|\cdot\|_{\Phi,\om})$ becomes a Banach $(L^1_\om(G), \|\cdot\|_{1,\om})$-bimodule under the standard convolution having $\Sm^\Phi_\om(G)$ as an essential $L^1_\om(G)$-submodule.
It is straightforward to verify that the mapping
\begin{align}\label{Eq:twist covn-weight conv-mapping}
\Lambda_\om: L^\Phi(G) \to L^\Phi_\om(G) \ , \ \Lambda_\om(f)=f/\om
\end{align}
is a linear isometric isomorphism satisfying $(f\in L^\Phi(G), g\in L^1(G))$
\begin{align}\label{Eq:twist covn-weight conv-relation}
\Lambda_\om(f\tw g)=\Lambda_\om(f)* \Lambda_\om(g) \ \ \ \text{and} \ \ \Lambda_\om(g\tw f)=\Lambda_\om(g)* \Lambda_\om(f),
\end{align}
where $\tw$ is the twisted convolution coming from the 2-coboundary
$$\Om(s,t)=\frac{\om(st)}{\om(s)\om(t)} \ \ \ (s,t\in G).$$
Moreover, similar relations to \eqref{Eq:twist covn-weight conv-mapping} and \eqref{Eq:twist covn-weight conv-relation} holds if we replace $L^\Phi(G)$ and $L^\Phi_\om(G)$ with $\Sm^\Phi(G)$ and $\Sm^\Phi_\om(G)$, respectively.
If, in addition, $(L^\Phi(G),\tw,\|\cdot\|_\Phi)$ is a Banach algebra, then \eqref{Eq:twist covn-weight conv-relation} holds for every
$f,g\in L^\Phi(G)$ so that  $(L^\Phi_\om(G),*, \|\cdot\|_{\Phi,\om})$ is a Banach algebra having $\Sm^\Phi_\om(G)$ as a closed subalgebra. Hence we have two alternative in viewing the algebraic structure on Orlicz spaces: either stay in the same space and use the twisted convolution or move to the weighted space but keep the product unchanged! For most part of this paper, it was more convenient for us to use the former approach. However, in few cases such as the one below, it is better to use the latter formulation.


\begin{thm}\label{T:Orlicz space-amen-untwisted}
Let $G$ be a discrete group, and let $\om$ be a weight on $G$. Suppose that  $(l^\Phi_\om(G),\|\cdot\|_{\Phi,\om})$ is a Banach algebra and $1/\om \in \Sm^\Psi(G)$. Then the following are equivalent:\\
$(i)$ $l^\Phi_\om(G)$ is amenable;\\
$(ii)$ $\Sm^\Phi_\om(G)$ is amenable;\\
$(iii)$ $G$ is finite.
\end{thm}

\begin{proof}
It is clear that if $G$ is finite, then $l^\Phi_\om(G)=\Sm^\Phi_\om(G)\cong l^1(G)$ so that they are amenable. We will now prove $(i)\Longrightarrow(iii)$. The proof of $(ii)\Longrightarrow(iii)$ is similar. Since $1/\om \in \Sm^\Psi(G)$, we have $l^\Phi_\om(G)\subseteq l^1(G)$. Now consider the augmentation ideal of $l^1(G)$, i.e.
$$I=\{f\in l^1(G): \sum_{s\in G} f(s)=0 \}$$
and put
$$I_\Phi:=I\cap l^\Phi_\om(G).$$
It is routine to verify that $I_\Phi$ is a closed two-sided ideal of codimension 1 in $l_\om^\Phi(G)$. Moreover, since $l^\Phi_\om(G)$ is amenable, by \cite[Theorem 2.3.7]{Run}, $I_\Phi$ has a bounded approximate identity $\{f_i\}$. Suppose that $f\in l^\Phi_\om(G)=\Sm_{\om^{-1}}^\Phi(G)^*$ is a cluster point of $\{f_i\}$.
Since $1/\om \in \Sm^\Psi(G)$, it follows that $f\in I_\Phi$ as
$$ \la f , 1 \ra=\la f\om , 1/\om \ra=\lim_{j\to \infty} \la f_{i_j}\om , 1/\om \ra=\lim_{j\to \infty} \la f_{i_j} , 1 \ra=0.$$
On the other hand, we have
$$ (\delta_s-\delta_e)*f=w^*-\lim_{i\to \infty} (\delta_s-\delta_e)*f_i=\delta_s-\delta_e \ \ \ (s\in G).$$
Thus, in particular, $f\in I_\Phi$ is nonzero and satisfies
$$f(e)-f(s)=1 \ \ \ (s\in G, s\neq e).$$
However, this is only possible if $G$ is finite.
\end{proof}

\subsection{Connes amenability}

Suppose that $A=\mathcal{B}^*$ is a dual Banach algebra. A dual Banach $A$-bimodule $X$
is {\it normal} if for every $x\in X$, the module multiplications
$$a \mapsto ax \ \ \text{and}\ \ a\mapsto xa $$
from $A$ into $X$ are $w^*$-$w^*$ continuous. We say that $A$ is {\it Connes-amenable} if every $w^*$-$w^*$ continuous derivation from $A$ into any normal dual Banach $A$-bimodule is inner \cite[Definition 4.4.7]{Run}.
The concept of Connes-amenability has proven to be an appropriate notion of amenability considered for dual Banach algebras. For example, it is well-known that a von Neumann algebra is amenable if and only if it is subhomogeneous, a rather restrictive condition \cite[Theorem 6.1.7]{Run}. However, it is Connes-amenable if and only if it is injective (this is a deep result in operators algebra, see \cite[Chapter 6]{Run}). M. Daws has also defined a notion of injectivity for unital dual Banach algebras and has shown that the corresponding results also holds in this category \cite[Theorem 6.13]{Daws}. It is also shown by Runde that the measure algebra of a locally compact group is Connes-amenable if and only if the underlying group is amenable thus providing an analogous of Jonshon's theorem for group algebras
\cite[Theorem 4.4.13]{Run}.

In Section \ref{S:Dual Banach algebras}, we have provided several classes of twisted Orlicz algebras which are dual Banach algebra. We can thus look at whether they are Connes-amenable. As we will see below, this rarely happen. This is similar to what we obtain in the preceding section with regard to the amenability. For the case of discrete groups, again our method only works for the untwisted cases.

\begin{thm}\label{T:Orlicz space-conn amen}
Let $G$ be a locally compact unimodular group, let $\Om\in \Zbw$, and let $\om$ be the weight associated to $|\Om|$. Suppose that there exit non-negative measurable functions $u,v \in \Sm^\Psi(G)$ satisfying \eqref{Eq:2-cocycle bdd sum}. Then:\\
$(i)$ If $G$ is non-discrete, then $(L^\Phi(G),\tw)$ is not Connes amenable;\\
$(ii)$ If $G$ is discrete and $1/\om \in \Sm^\Psi(G)$, then $(l^\Phi_\om(G),*)$ is Connes amenable if and only if $G$ is finite.
\end{thm}

\begin{proof}
The statement in (i) follows from Theorem \ref{T:Orlicz space-bai-non discrete}(iv) and the fact that Connes-amenability implies that the algebra should be unital. For (ii), let $I_\Phi$ be the augmentation ideal in $l^\Phi(G)$ defined in the proof of Theorem \ref{T:Orlicz space-amen-untwisted} which is shown to be a $w^*$-closed. Hence its unitization can be identified with $l^\Phi(G)$
as dual Banach algebra, and so, by our assumption and \cite[Proposition 6.1]{Daws}, $I$ is Connes-amenable. In particular, it is unital so that, as it is shown in the proof of Theorem \ref{T:Orlicz space-amen-non discrete}, $G$ must be finite.
\end{proof}

\end{document}